\documentclass[10pt]{article}
\usepackage{color}
\usepackage{amssymb}
\usepackage{amsthm,array,amssymb,amscd,amsfonts,latexsym, url}
\usepackage{amsmath}
\usepackage[all]{xy}
\newtheorem{theo}{Theorem}[section]
\newtheorem{prop}[theo]{Proposition}
\newtheorem{claim}[theo]{Claim}
\newtheorem{lemm}[theo]{Lemma}
\newtheorem{sublemm}[theo]{Sublemma}
\newtheorem{coro}[theo]{Corollary}
\newtheorem{rema}[theo]{Remark}
\newtheorem{Defi}[theo]{Definition}

\newtheorem{question}[theo]{Question}

\title{Geometric  representability  of $1$-cycles on rationally connected threefolds}
\author{Claire Voisin\footnote{The author is supported by the ERC Synergy Grant HyperK (Grant agreement No. 854361).}}

\date{}

\newfont{\gothic}{eufb10}
\begin{document}

\maketitle
\setcounter{section}{-1}
\begin{flushright} {\it \`A la m\'{e}moire d'Alberto Collino}
\end{flushright}

\begin{abstract} We prove that for any rationally connected threefold $X$, there exist a smooth projective surface $S$ and a family of $1$-cycles on $X$ parameterized by $S$, inducing an Abel-Jacobi  isomorphism ${\rm Alb}(S)\cong J^3(X)$. This statement was  previously known for some classes of smooth   Fano threefolds. We prove a similar result for the Walker Abel-Jacobi map on  $1$-cycles on higher dimensional rationally connected manifolds.
 \end{abstract}
\section{Introduction}
Non-representability phenomena for algebraic cycles on smooth projective complex varieties were discovered by Mumford \cite{mumford}, Griffiths \cite{griffiths} and Clemens \cite{clemens}.
If $X$ is a smooth complex projective variety, for each integer $k\geq0$, the $k$-th  Griffiths Abel-Jacobi map  constructed by transcendental methods  is a group morphism
$$\Phi_X^k: {\rm CH}^k(X)_{\rm hom}\rightarrow J^{2k-1}(X),$$
where $$J^{2k-1}(X):=H^{2k-1}(X,\mathbb{C})/(F^kH^{2k-1}(X)\oplus H^{2k-1}(X,\mathbb{Z})_{\rm tf})$$ is a  (usually non-algebraic)  complex torus and ${\rm CH}^k(X)_{\rm hom}$ is the group of codimension $k$ cycles   homologous to zero on $X$ modulo rational equivalence. (We use above  the notation $\Gamma_{\rm tf}$ for $\Gamma/{\rm Tors}(\Gamma)$, for any abelian group  $\Gamma$.)
Denoting by  ${\rm CH}^k(X)_{\rm alg}\subset  {\rm CH}^k(X)_{\rm hom}$ the subgroup of cycles algebraically equivalent to $0$ on $X$,  the image $J^{2k-1}(X)_{\rm alg}:=\Phi_X^k({\rm CH}^k(X)_{\rm alg})$   is an abelian variety (we refer to \cite{casalaina} for   an algebraic construction of it in the spirit of Murre \cite{murre} for codimension $2$ cycles  and Serre \cite{serre} for $0$-cycles).
The group ${\rm CH}^k(X)_{\rm alg}$ is not an abstract group; it has a sort of algebraic structure since it is exhausted by the images of morphisms
$$Z_*:M\rightarrow {\rm CH}^k(X)_{\rm alg},\,m\mapsto Z_m-Z_{m_0}$$
induced by any codimension $k$-cycle $Z\subset M\times X$, with $M$ smooth  and connected, equipped with a reference point $m_0$. Here the smoothness of $M$ is used to define $Z_*$, thanks to Fulton refined intersections \cite{fulton},  and connectedness is required to guarantee that the considered cycles are algebraically equivalent to $0$.
For each such cycle, if $M$ is furthermore projective,  the composition $\Phi_X^k\circ Z_*: M\rightarrow J^{2k-1}(X)_{\rm alg}$ induces by the universal property of the Albanese map (see \cite{serre}) a morphism
$$\Phi_Z: {\rm Alb}(M)\rightarrow J^{2k-1}(X)_{\rm alg}$$
such that $\Phi_Z\circ{\rm alb}_M(m-m_0)=\Phi_X^k\circ Z_*(m)$.
\begin{rema}\label{remapbplace} {\rm Note that  there exists a universal  codimension $k$ cycle $Z_J$  {\it with $\mathbb{Q}$-coefficients} on $J^{2k-1}(X)_{\rm alg}\times X$, that is,  a cycle $Z_J\in{\rm CH}^k(J^{2k-1}(X)_{\rm alg}\times X)$ such that
 \begin{eqnarray}\label{eq9juillet}\Phi_{Z_J}=N Id_{J^{2k-1}(X)_{\rm alg}}: {\rm Alb}(J^{2k-1}(X)_{\rm alg})=J^{2k-1}(X)_{\rm alg}\rightarrow  J^{2k-1}(X)_{\rm alg}
 \end{eqnarray}
 for some integer $N>0$. To see this, one uses the fact that, almost by definition of  $J^{2k-1}(X)_{\rm alg}$, there exists a smooth projective variety $M$ and a codimension $k$ algebraic cycle $Z_M\in{\rm CH}^k(M\times X)$ such that
 $$\psi:=\Phi_{Z_M}\circ {\rm alb}_M:M\rightarrow J^{2k-1}(X)_{\rm alg}$$
 is surjective. Replacing $M$ by a complete intersection of hypersurfaces in general position, one can assume that $\psi$ is generically finite of degree $N$. Then $$Z_J:=(\psi,Id_X)_*Z_M$$ satisfies (\ref{eq9juillet}).}
 \end{rema}

Since the seminal work of Mumford \cite{mumford}, we know that   the group ${\rm CH}^k(X)_{\rm alg}$ is not representable in general, where by representability of ${\rm CH}^k(X)_{\rm alg}$, we mean here that there exist  a connected projective manifold $M$ and a codimension $k$ cycle $Z\in{\rm CH}^k(M\times X)$ such that the induced map
$$M\rightarrow  {\rm CH}^k(X)_{\rm alg},$$
$$m\mapsto Z_*(m-m_0)$$
is surjective, where $m_0$ is a chosen reference point of $M$. The typical Mumford  examples of nonrepresentable Chow groups are given by  the groups  ${\rm CH}_0(S)$ of  smooth projective surfaces $S$ with $p_g\not=0$. Indeed, Mumford \cite{mumford} shows that ${\rm CH}_0(S)$ is under this assumption infinite dimensional in the sense that it cannot be exhausted by countably many  families of $0$-cycles parameterized by varieties of bounded dimension. As an easy consequence, the Albanese map (or Abel-Jacobi map) is not injective on ${\rm CH}_0(S)_{\rm alg}$.

The Abel-Jacobi map $\Phi_X^k$ is also not surjective in general (see \cite{griffiths}) and the fact  that  the transcendental part of the Abel-Jacobi map, defined on cycles modulo algebraic equivalence, with value in $J^{2k-1}(X)/J^{2k-1}(X)_{\rm alg}$, can be nontrivial,  is a celebrated  discovery  of Griffiths. Clemens \cite{clemens} even proved that its  image can be non-finitely generated modulo torsion,  illustrating another aspect of nonrepresentability of cycles of intermediate dimension. However, in certain cases, for example rationally connected varieties $X$, or more generally  smooth projective varieties with trivial ${\rm CH}_0$ group, one can easily show using the Bloch-Srinivas decomposition of the diagonal (see \cite{blochsrinivas})  that the two intermediate Jacobians $J^3(X)$, which correspond to codimension $2$ cycles, and   $J^{2n-3}(X)$, which correspond to $1$-cycles,  have no transcendental part. Nevertheless,  the geometry of the Abel-Jacobi map remains  mysterious even in these cases.
We study in this paper the following  notion of ``geometric  representability"  for $J^{2k-1}(X)$, which  implies and strengthens the condition  that  $J^{2k-1}(X)=J^{2k-1}(X)_{\rm alg}$.

\begin{Defi} \label{defirep} The group $ J^{2k-1}(X)$ is geometrically representable if there exist  a smooth projective variety
$M$  and a codimension $k$ cycle $Z_M\in{\rm CH}^k(M\times X)$ such that
$$\Phi_{Z_M}:{\rm Alb}(M)\rightarrow  J^{2k-1}(X)$$
is an isomorphism.
\end{Defi}
  One motivation for this definition is the following  result proved in \cite[Proposition 3.6]{voisinconiveau}:
\begin{prop} Let $X$ be stably rational of dimension $n$. Then $J^3(X)$ and $J^{2n-3}(X)$ are geometrically representable in the sense of Definition \ref{defirep}. More generally, the geometric representability of
 $J^3(X)$ (resp. $J^{2n-3}(X)$)  is a  stably birationally invariant property of  smooth projective varieties $X$.
\end{prop}
\begin{rema} {\rm  The intermediate Jacobian   $J^{2n-3}(X)$ corresponds to  $1$-cycles on $X$. By blowing-up  a smooth subvariety $W$  of  codimension at least  $2$ in  $X$, we add to the group ${\rm CH}_1(X)$ the group ${\rm CH}_0(W)$. If ${\rm dim}\,X\geq 4$, we have ${\rm dim}\,W\geq 2$, and  $W$ can be chosen to be a surface with $p_g\not=0$. It follows that the group  ${\rm CH}_1(X)$, for  a  rational  variety $X$ of dimension $\geq 4$  is not in general  representable in the  sense of Mumford.  Another example of this phenomenon, which is   a Fano example with Picard number $1$,  is given by the cubic fourfolds, for which the group of $1$-cycles is a direct summand in the ${\rm CH}_0$-group of a surface and is not representable in the Mumford sense. Some smooth cubic fourfolds are rational.}
\end{rema}
\begin{rema}{\rm  We proved  in \cite{voisinuniv} that,  even in  the case of $1$-cycles on rationally connected threefolds, there   does not necessarily exist a universal codimension $k$  cycle on $J^{2k-1}(X)_{\rm alg}\times X$, namely a cycle $Z_J\in {\rm CH}^k(J^{2k-1}(X)_{\rm alg}\times X)$  which satisfies  (\ref{eq9juillet}) with $N=1$.  This shows that we cannot in general ask that $M=J^{2k-1}(X)$ in Definition \ref{defirep}. Furthermore, the existence of  a universal codimension $k$  cycle is not a birationally invariant property for $k\geq 3$, and even for  $k=n-1$ ($1$-cycles), see Corollary \ref{corobirex}. }
\end{rema}
We are going to focus in this paper on the case of $1$-cycles on rationally connected manifolds. The case of $1$-cycles on  rationally connected threefolds  is  special since   the  group of codimension $2$ cycles on rationally connected manifolds (or manifolds with trivial ${\rm CH}_0$-group)  always satisfies representability in the Mumford sense by the following
\begin{theo}\label{blochsri} (Bloch-Srinivas \cite{blochsrinivas}.) The Abel-Jacobi map $\Phi_X^2:{\rm CH}^2(X)_{\rm alg}\rightarrow J^3(X)$ is an isomorphism for codimension $2$ cycles  on  any smooth projective variety $X$ with ${\rm CH}_0(X)=\mathbb{Z}$, hence in particular  on any smooth projective  rationally connected  manifold $X$.
\end{theo}

 We prove in this paper the following  result which answers a question asked by Clemens  in \cite{clemens2} (see also \cite{voisinconiveau}).
\begin{theo} \label{theomain} Let $X$ be a smooth projective  rationally connected $3$-fold over $\mathbb{C}$. Then there exist  a smooth complex projective manifold $M$ and a  family of $1$-cycles
$$Z\in {\rm CH}^{2}(M\times X)$$
such that the induced Abel-Jacobi map
\begin{eqnarray}\label{isoAJRC3folds} \Phi_Z: {\rm Alb}(M)\rightarrow J^3(X)
\end{eqnarray}
is an isomorphism. In other words, $J^3(X)$ is geometrically representable in the sense of Definition \ref{defirep}.
\end{theo}

\begin{rema} {\rm If one replaces ``isomorphism" by ``isogeny" in Theorem \ref{theomain}, then the statement holds for codimension $k$-cycles and with $M=J^{2k-1}(X)$, under the only  assumption on $X$ that $J^{2k-1}(X)=J^{2k-1}(X)_{\rm alg}$ , see Remark \ref{remapbplace}.}\end{rema}
Combining Theorem \ref{theomain} and Theorem \ref{blochsri}, we conclude
 \begin{coro} If $X$ is a rationally connected threefold, there is a codimension $2$ cycle $Z\in{\rm CH}^2(M\times X)$ such that
 $$Z_*:{\rm CH}_0(M)_{\rm hom}\rightarrow {\rm CH}^2(X)_{\rm alg}$$
 factors through an isomorphism
 $${\rm Alb}(M)\cong {\rm CH}^2(X)_{\rm alg}.$$
 \end{coro}

The proof of Theorem \ref{theomain} will be presented in Section \ref{secproofth}. By a small variant of the proof, we will also establish in Section \ref{sectheowalker}  the following generalization to the case of $1$-cycles on rationally connected manifolds of any dimension.
\begin{theo}\label{theoWalker}    Let $X$ be a smooth projective  rationally connected $n$-fold over $\mathbb{C}$. Then there exist a smooth complex projective manifold $M$ and a  family of $1$-cycles
$$Z\in {\rm CH}^{n-1}(M\times X)$$
such that the induced  Walker Abel-Jacobi map
\begin{eqnarray}\label{isoAJRCnfolds} \Phi_Z^{\rm Walker}:= \Phi_{X,{\rm Walker}}^{n-1}\circ Z_*: {\rm Alb}(M)\rightarrow J^{2n-3}(X)_{\rm Walker}
\end{eqnarray}
is an isomorphism.
\end{theo}
 The Walker Abel-Jacobi map  is a natural  factorization of the Griffiths Abel-Jacobi map on codimension $k$ cycles algebraically equivalent to zero  through an abelian variety which is isogenous to $J^{2k-1}(X)_{\rm alg}$. We refer to Section \ref{sectheowalker} for a discussion.
Coming back to Theorem \ref{theomain}, let us make a few remarks.
\begin{rema}{\rm We can even assume  in the conclusion of  Theorem \ref{theomain}   that $M$ is  a smooth projective surface. Indeed, by the Lefschetz theorem on hyperplane sections, if $S\subset M$ is a smooth surface  which is a complete intersection of ample hypersurfaces,
the natural map ${\rm Alb}(S)\rightarrow {\rm Alb}(M)$ is an isomorphism.}
\end{rema}
\begin{rema} {\rm We cannot in general assume that $M$ is a smooth curve. In many cases, the intermediate Jacobian of a rationally connected threefold has a unique principal polarization, so if it were isomorphic to the Jacobian of a curve, it would be isomorphic  as a principally polarized abelian variety to the Jacobian of a curve. This is the well-known obstruction to rationality discovered by Clemens and Griffiths \cite{clemensgriffiths}, and it  is already nontrivial  in the case of the cubic threefold.}
\end{rema}

Theorem \ref{theomain} had been proved earlier by explicit constructions for some classes of Fano threefolds.  Clemens-Griffiths \cite{clemensgriffiths} studied the case of the cubic threefold $X$ and proved that the family  of lines in $X$, that is classically known to be parameterized by  a smooth surface $\Sigma_X$, induces an Abel-Jacobi isomorphism between ${\rm Alb}(\Sigma_X)$ and $J^3(X)$.  The paper \cite{welters} treats the case of the general quartic double solid $r:X\rightarrow \mathbb{P}^3$, and establishes a similar Abel-Jacobi isogeny for its surface of ``lines'', namely degree $1$ curves on $X$ with respect to the ample line bundle  $r^*\mathcal{O}_{\mathbb{P}^3}(1)$, and later on, this has been improved by Clemens \cite{clemens1} who proved that the Abel-Jacobi morphism for the family of lines in the quartic double solid   is an isomorphism.  Clemens'  strategy has been then adapted successfully  by  Letizia in \cite{letizia} to establish  the corresponding   statement  for a general quartic threefold  in $\mathbb{P}^4$ and its surface of conics, that had been shown to be smooth in \cite{collino2}. A similar result was also obtained by Ceresa and Verra \cite{ceresaverra} who treated the case of the sextic double solid.   Although it does not concern $1$-cycles, we also mention a similar result  by Collino, concerning $2$-cycles in a general cubic $5$-fold $X$. Collino proves in \cite{collino1} that the surface $\Sigma_X$ of planes is smooth and induces an Abel-Jacobi isomorphism ${\rm Alb}(\Sigma_X)\rightarrow J^5(X)$.

Theorem  \ref{theomain} is an improvement of  Theorem  0.2  of \cite{voisinconiveau}, which  is the following statement
\begin{theo}\label{theogeotop} (Voisin \cite{voisinconiveau}.) Let $X$ be a rationally connected $3$-fold. There exist a smooth projective curve
$C$ and a cycle $Z\in {\rm CH}^{2}(C\times X)$ such that
$$ \Phi_Z: J(C)\rightarrow J^{3}(X)
$$
is surjective with  connected fibers. (Equivalently, the induced morphism  $\Phi_{X*}: H_1(C,\mathbb{Z})\rightarrow H_3(X,\mathbb{Z})_{\rm tf}$ is surjective.)
\end{theo}
Indeed, Theorem \ref{theomain} implies   Theorem \ref{theogeotop} since, assuming  Theorem \ref{theomain} and  choosing a smooth  curve $C\subset M$ which is a complete intersection of ample hypersurfaces, the natural morphism $J(C)\rightarrow {\rm Alb}(M)$ has connected fibers by the Lefschetz hyperplane section theorem which says that the morphism $H_1(C,\mathbb{Z})\rightarrow H_1(M,\mathbb{Z})$ is surjective. Then the restricted cycle $Z_{M\mid C\times X}\in {\rm CH}^{2}(C\times X)$ satisfies the conclusion of Theorem \ref{theogeotop}.

One consequence of Theorem \ref{theomain}  is the fact that the obstruction to the existence of a universal codimension $2$ cycle on a rationally connected threefold $X$ always  reduces to the  obstruction to the existence of a universal $0$-cycle parameterized by ${\rm Alb}(M)$  for some smooth projective variety $M$:
\begin{coro} \label{coroexnounivintro}  (See Corollary \ref{coroplustard}.) There exists  a  smooth projective variety $M$  such that $M$ does not admit a universal $0$-cycle.
\end{coro}
We note also the following
\begin{coro} \label{corobirex} The existence of a universal codimension $k$-cycle is not  a birationally invariant property for $k\geq 3$.
\end{coro}
Indeed, if we  embed a surface $M$ as in Corollary \ref{coroexnounivintro} in a  blow-up $\widetilde{\mathbb{P}^4}$ of  $\mathbb{P}^4$ at finitely many points, and then  blow-up $\widetilde{\mathbb{P}^4}$ along $M$, we get by the blow-up formulas that a universal $1$-cycle on   $\widetilde{\mathbb{P}^4}$ would provide a  universal $0$-cycle on  $M$. Hence it does not exist in general.

We will also explain in Section \ref{sectionopen}  the following  formal consequence concerning  the Albanese motive,  of the existence of smooth projective varieties with no universal $0$-cycle, (for example  the surface of lines of a general quartic double solid, and more generally, any variety $M$ as in Theorem \ref{theomain}, if $X$ does not admit a universal codimension $2$ cycle). It is proved in \cite{voisinconiveau} (see also  Lemma \ref{levenantdeconiveau}) that given a smooth projective variety $M$ and an abelian variety $A$ contained  in
${\rm Alb}(M)$,  there exist  a smooth projective variety $N$ and a $0$-correspondence from  $N$ to  $M$ such that
$\Gamma_*: {\rm Alb}(N)\rightarrow {\rm Alb}(M) $ is an isomorphism from $ {\rm Alb}(N)$ onto $A$. We will show that this result cannot be extended to direct sums, in the following sense
\begin{prop}\label{corosursdir} There exists a smooth projective variety $M$ such that

(1)   ${\rm Alb}(M)$ decomposes as a direct sum ${\rm Alb}(M)=A\oplus B$,

(2) there does not exist a smooth projective variety
$N$ and two $0$-correspondences $\Gamma_1$ from $N$ to $M$ and $\Gamma_2$ from $M$ to $N$ such that
$\Gamma_{1*}:{\rm Alb}(N)\rightarrow  {\rm Alb}(M)$ has for image $A$ and
$$\Gamma_{2*}\circ \Gamma_{1*}=Id: {\rm Alb}(N)\rightarrow  {\rm Alb}(N).$$
\end{prop}
In other words, a direct sum decomposition of the Albanese variety of $M$ does not  always come from a projector in ${\rm CH}(M\times M)$.

We prove Theorem \ref{theomain} in Section \ref{secproofth} starting from Theorem \ref{theogeotop}. Note however that  the variant of the proof that we will use to prove Theorem \ref{theoWalker}  in Section \ref{sectheowalker}  does not rely on Theorem  \ref{theogeotop} and instead gives an alternative proof of it. Section \ref{sectionopen} discusses two stronger   notions of representability for $1$-cycles, namely Shen's notion of universally generating family on one hand (we do not know whether it is always satisfied, see Question \ref{questioninteressante}) and the existence of  a universal codimension $2$ cycle on the other hand (we know that it is not always satisfied, see \cite{voisinuniv}). One result that we will prove as a consequence of Theorem \ref{theomain} is
\begin{theo} (See Theorem \ref{theoshenunivalbiso}) Let $X$ be  a rationally connected threefold. Assume that $X$ admits a universally generating family $(N,Z_N)$. Then $X$ admits a universally generating family $(M,Z_M)$ such that $\Phi_{Z_M}:{\rm Alb}(M)\rightarrow J^3(X)$ is an isomorphism.
\end{theo}

\section{Proof of Theorem \ref{theomain}\label{secproofth}}
\subsection{Preliminary results}
Let $X$ be a rationally connected threefold. By  Theorem \ref{theogeotop}, we know that there exist  a
smooth projective curve $C$ and a codimension $2$-cycle $Z\in {\rm CH}^{2}(C\times X)$ such that
the  Abel-Jacobi map
$\Phi_Z: J(C)\rightarrow J^{3}(X)$ has connected fibers, that is, its kernel $K:={\rm Ker}\, \Phi_Z$ is an abelian variety.
 We will prove the following result
\begin{theo} \label{theonewform}  There exists  a smooth projective variety $M$ carrying a codimension  $2$ cycle $Z_M\in{\rm CH}^{2}(M\times X)$, with an inclusion
$j: C\hookrightarrow M$ such that

(1)  $\Phi_{Z_M}\circ j_*=\Phi_Z$ and

(2)
${\rm Ker}\, \Phi_Z= {\rm Ker}\,(j_*: J(C)\rightarrow {\rm Alb}(M))$.
\end{theo}
\begin{claim}  Theorem \ref{theomain} follows from Theorem \ref{theonewform}.
\end{claim}
\begin{proof}  As ${\rm Ker}\, \Phi_Z= {\rm Ker}\,(j_*: J(C)\rightarrow {\rm Alb}(M))$, the image ${\rm Im}\,(j_*: J(C)\rightarrow {\rm Alb}(M))$ is isomorphic to $J^{3}(X)$.  It follows that $J^{3}(X)\cong {\rm Im}\,j_*\subset {\rm Alb}(M)$ is  realized as an abelian subvariety of $ {\rm Alb}(M)$ and we can then apply  the following
\begin{lemm} \label{levenantdeconiveau} (See \cite[Lemma 3.7]{voisinconiveau}.)  Let $M$ be a smooth projective variety and let $A\subset {\rm Alb}(M)$ be an abelian subvariety. Then there exist a smooth projective variety $M'$ and a $0$-correspondence $\Gamma\in{\rm CH}^{{\rm dim}\,M}(M'\times M)$, such that $\Gamma_*:{\rm Alb}(M')\rightarrow {\rm Alb}(M)$ is the composition of an  isomorphism between ${\rm Alb}(M')$  and $A$ and the inclusion of $A$ in ${\rm Alb}(M)$.
\end{lemm}
We apply Lemma \ref{levenantdeconiveau} to $A:=J^3(X)\subset {\rm Alb}(M)$. By Theorem \ref{theonewform},   the morphism
$\Phi_{Z_M}: {\rm Alb}(M)\rightarrow J^{3}(X)$ is an isomorphism when restricted to $J^{3}(X)\cong {\rm Im}\,j_* \subset {\rm Alb}(M)$, so we conclude that the cycle $Z_{M'}:=Z_M\circ \Gamma\in {\rm CH}^{2}(M'\times X)$ satisfies the conclusion of Theorem \ref{theomain}.
\end{proof}
    For the proof of Theorem \ref{theonewform}, we will  use the  following lemma.
\begin{lemm}  \label{profacilepourtuer}  Let $C\subset J(C)$ be a smooth projective curve imbedded in its Jacobian, and let $K\subset J(C)$ be an abelian subvariety. Let $x\in K$ be a very general point.
Let $M$ be a smooth projective variety containing the curve $C$ and let $j:C\hookrightarrow M$ be the   inclusion, inducing $j_*:J(C)\rightarrow {\rm Alb}(M)$. Assume
\begin{eqnarray}\label{eqvandansALBM}  {\rm alb}_M \circ j_{*}(x)=0\,\,{\rm in}\,\, {\rm Alb}\,M.
\end{eqnarray} Then
$ {\rm alb}_M\circ j_*$ vanishes on $K$.
\end{lemm}
\begin{proof} As $x$ is very general in $K$, the subgroup generated by $x$ is Zariski dense in $K$. As  $ {\rm alb}_M \circ j_{*}$ vanishes on this subgroup, it vanishes on $K$.
\end{proof}

\subsection{Proof of Theorem \ref{theonewform} \label{secprooftheonewform}}
Let $X$ be a rationally connected threefold and let
 $C$ be a smooth projective curve, equipped with a   codimension $2$-cycle $Z_C\in {\rm CH}^{2}(C\times X)$ such that
the  Abel-Jacobi map
$\Phi_Z: J(C)\rightarrow J^{3}(X)$ has connected fibers.
In order to construct a pair  $(M, Z_M)$ satisfying the conclusion of Theorem \ref{theonewform}, we will first use the following observation,  already used in \cite{voisinJAG} and also in \cite{voisinconiveau}.
\begin{lemm}  \label{lenoyau} There exists a codimension $2$-cycle $Z_{J(C)}$ on $J(C)\times X$, such that
\begin{eqnarray}\label{equationJCC} \Phi_{Z_{J(C)}}=\Phi_Z: J(C)\rightarrow J^3(X).\end{eqnarray}
Furthermore, for any $x\in K={\rm Ker}\,(\Phi_Z:J(C)\rightarrow J^{3}(X))$, one has
\begin{eqnarray} \label{eqvantardive}  Z_{J(C)*}(\{x\}-\{0_J\})=0\,\,{\rm in}\,\, {\rm CH}^2(X).
\end{eqnarray}
\end{lemm}
\begin{proof} The  cycle $Z_{J(C)}$  is obtained by seeing $J(C)$ as birational to $C^{(g)}$, $g:=g(C)$, and by descending to $C^{(g)}\times X$ the  cycle
$$ \sum_i{\rm pr}_i^*Z_C$$
which is invariant under the natural action of the  symmetric group $\mathfrak{S}_g$  on $C^g\times X$, where in the formula above, the  ${\rm pr}_i:C^g\times X\rightarrow C\times X$ are the obvious projections. This provides a cycle in $C^{(g)}\times X$ which we further descend to $J(C)\times X$ via the birational morphism $C^{(g)}\rightarrow J(C)$. By construction of  $Z_{J(C)}$, one has
for a general element  $x_\cdot=x_1+\ldots+x_g\in C^{(g)}$,
\begin{eqnarray} \label{eqvantardivell}Z_{J(C)*}(x_{\cdot})=Z_{x_1}+\ldots+Z_{x_g}\,\,{\rm in}\,\,{\rm CH}^2(X).
\end{eqnarray}
By applying  $\Phi_X$ to (\ref{eqvantardivell}), we get (\ref{equationJCC}). It follows from (\ref{equationJCC})  that
$$\Phi_X( Z_{J(C)*}(\{x\}-\{0_J\}))=0\,\,{\rm in }\,\,J^3(X),$$
which implies (\ref{eqvantardive}) by Theorem \ref{blochsri}.
\end{proof}

We will prove   the following theorem in Section \ref{secauxi}. The proof is in fact very similar to that of  Claim 2.13  of \cite{voisinconiveau}. We choose as before a very general point $x\in K={\rm Ker}\,\Phi_Z={\rm Ker}\,\Phi_{Z_J}$.
\begin{theo}\label{theoauxi} There exist  a singular quasiprojective variety  $W:=W_1\cup W_2\cup W_3$, where
$W_1$ is a Zariski open set of $J(C)$ containing $x$ and $0_J$,  and a diagram
 \begin{eqnarray}\label{eqcorresp} \begin{matrix}
 & \Sigma&\stackrel{f}{\rightarrow}& X
 \\
&p\downarrow& &
\\
&W& &
\end{matrix}
\end{eqnarray}
with the following properties.
\begin{enumerate}
\item \label{item1} $p$ is proper flat of relative dimension $1$ with semistable fibers.
\item \label{item2} The restriction to $W_1\subset  W$ of the correspondence (\ref{eqcorresp}) has for induced Abel-Jacobi map $W_1\rightarrow J^3(X)$  the restriction to $W_1$ of the morphism $\Phi_{Z_J}$.
\item \label{item3} $W_2$ is isomorphic to a Zariski open set of $\mathbb{P}^1$, and  $W_3$ is Zariski open in a smooth projective connected and simply connected variety. The open subset  $W_2\subset \mathbb{P}^1$ contains the points $0,\,\infty \in \mathbb{P}^1$ and  intersects $J(C) $ in $0_J\in J(C)$ and $W_3$ in a point $w_2\in W_3$, at the  respective points  $0,\,\infty\in W_2$. $W_3$ intersects $J(C)$  at the point $x\in J(C)$, and we denote by $w_3\in W_3$ the  corresponding point.
\item \label{4} The fibers $f_t:\Sigma_t\rightarrow X$ are stable maps which are unobstructed at the three   singular points $0_J,\,x,\,w_2$ of $W=W_1\cup W_2\cup W_3$.
\end{enumerate}
\end{theo}
 The statement and construction would be  slightly clearer if $W_2\cup W_3$ were a chain of rational curves meeting $W_1$ in the two points $0_J$ and $x$. Unfortunately  we have not been able to achieve this. The heuristic proof of Theorem \ref{theoauxi}, \ref{item1} to \ref{item3} is the fact that, by Lemma \ref{lenoyau}, we know that $Z_{J(C)*}(\{x\}-\{0_J\})=0$ in ${\rm CH}_1(X)$. Assume naively that this vanishing  translates into the existence of
  a family $Z'$ of curves parameterized by $\mathbb{P}^1$ such that $Z'_0=Z_{0_J},\,Z'_\infty=Z_{x}$, then we would take
 $W_2=\mathbb{P}^1$, $W_3=\emptyset$ and the correspondence $\Sigma$ would be given by $Z_J$ on $W_1$, and $Z'$ on $\mathbb{P}^1$, glued to $J$ via $0=0_J,\,x=\infty$. To achieve \ref{4}, one uses the Koll\'ar-Miyaoka-Mori technic \cite{komimo} which allows to make curves unobstructed by adding trees of very free rational curves. One has to show that this can be done in families at least generically, and without changing the Abel-Jacobi map.

We will explain the technical details of the proof of Theorem \ref{theoauxi} in next section and we  now prove Theorem \ref{theonewform} assuming Theorem  \ref{theoauxi}.
\begin{proof}[Proof of Theorem \ref{theonewform}] We construct the smooth projective variety $M$ and the cycle $Z_M$ starting from the universal deformation of the family of stable maps $f_t: \Sigma_t\rightarrow X,\,t\in W$ appearing in Theorem \ref{theoauxi}. More precisely, by Theorem \ref{theoauxi}, \ref{4},  we conclude that this  deformation space $M_0$  is smooth near  the three singular points $0_J,\,x,\,w_2$ and more precisely, using the fact that $J(C),\,\mathbb{P}^1$ and $W_3$ are smooth and connected, $M_0$ has exactly one connected component  $M_1$ which  contains the image of $W$ under  the classifying map  associated with the family of stable maps (\ref{eqcorresp}) and is smooth near $0_J,\,x,\,w_2$. It follows  that a smooth projective desingularization $M$ of  $M_1$  contains a union
\begin{eqnarray}\label{eqbiratWWW} \widetilde{J(C)} \cup \widetilde{W_2}\cup \widetilde{W_3}
\end{eqnarray}
where the  varieties  $\widetilde{J(C)} $, $\widetilde{W_2}$ and $ \widetilde{W_3}$ are smooth projective birational to $J(C) $, $W_2$ and $W_3$ respectively, and the union (\ref{eqbiratWWW}) is  isomorphic to $J(C)\cup W_2\cup W_3$ near  the intersection points  $0_J$, $x$ and $w_2$.
The cycle $Z_{M}$ is given over a dense  Zariski open set of $M$  by the universal evaluation map

\begin{eqnarray}\label{eqcorrespM} \begin{matrix}
 & \Sigma_M&\stackrel{f_M}{\rightarrow}& X
 \\
&p\downarrow& &
\\
&M& &.
\end{matrix}
\end{eqnarray}

We now claim that  the morphism $j_{1*}:J(C)={\rm Alb}(\widetilde{J(C)})\rightarrow {\rm Alb}(M)$ induced by  the inclusion $j_1$ of $\widetilde{J(C)}$ in $M$  constructed above satisfies
\begin{eqnarray}\label{eqjx=0} j_*(x)=0.\end{eqnarray}
Indeed, denoting by $j_2:\widetilde{W_2}\rightarrow M$ and $j_3:\widetilde{W}_3\rightarrow M$ the restrictions of $f$  we have ${\rm alb}_M(j_{1}(x)-j_3(w_2))=0$ because $j_1(x)=j_3(w_3)$ and $\widetilde{W}_3$ is simply connected. Similarly,
${\rm alb}_M(j_{1}(0_J)-j_3(w_2))=0$ because $j_1(0_J)=j_2(0)$, $j_3(w_2)=j_2(\infty)$ and $\widetilde{W_2}\cong \mathbb{P}^1$ is simply connected. Thus
${\rm alb}_M(j_{1}(x)-j_{1}(0_J))=0$, which
 proves the claim.

We now apply Lemma  \ref{profacilepourtuer} to the inclusion $j$ of $\widetilde{J(C)}$ in $M$ and conclude that
$K={\rm Ker}\,\Phi_Z$ is contained in ${\rm Ker}\,(j_*:J(C)\rightarrow {\rm Alb}(M))$. Finally, as
$\Phi_{Z_M}\circ j_*=\Phi_{Z_J}$  by statement \ref{item2} of Theorem \ref{theoauxi},  we also  have  ${\rm Ker}\,(j_*:J(C)\rightarrow {\rm Alb}(M))\subset K$ and
   the proof of Theorem \ref{theonewform} is concluded.
\end{proof}
\subsection{Proof of Theorem \ref{theoauxi} \label{secauxi}}

We describe in this section the proof of Theorem \ref{theoauxi}. It is obtained by a minor modification of the proof of Theorem \ref{theogeotop}  of  \cite{voisinconiveau}. This proof being rather intricate, we will   sketch its main steps, explaining  where and how   it has to be modified in order  to  give  a proof of Theorem \ref{theoauxi}.
We recall  that Lemma \ref{lenoyau} provides us with a codimension $2$ cycle $Z_{J(C)}$ on $J(C)\times X$ with the property that
${\rm Ker}\,(\Phi_{Z_{J(C)}}: J(C)\rightarrow J^3(X))$ is an abelian variety $K$. The cycle $Z_{J(C)}$ can be modified by adding a constant $1$-cycle in $X$ and we can thus assume, by choosing  a rationally equivalent representative, that over a Zariski open set $W_1$ of $J(C)$,  the fiber  $Z_t,\,t\in W_1$,  is  a smooth curve in $X$.  We can also arrange that the  Zariski open set $W_1$ contains the  two points $x,\,0_J\in K\subset J(C)$.  The cycle $Z_{J(C)}$ is thus represented over $W_1$ by a family of smooth curves (or stable maps)

\begin{eqnarray} \label{eqstablemap1}   \begin{matrix}
 & \mathcal{C}_1 &\stackrel{f_1}{\rightarrow}& X
 \\
&p_1\downarrow& &
\\
&W_{1}& &
\end{matrix}
\end{eqnarray}
in $X$.

\vspace{0.5cm}

{\bf Step 1}. {\it Constructing $W_2$  and the   family of stable maps over it.}  By assumption, $x\in {\rm Ker}\,(\Phi_{Z_J}:{\rm Alb}(J(C))=J(C)\rightarrow J^3(X))$, which gives  $$\Phi_{Z_J}({\rm alb}_{J(C)} (\{x\}-\{0_J\}))=\Phi_X(\mathcal{C}_{1,x}-\mathcal{C}_{1,0_J})=0.$$
 We thus get, using Theorem \ref{blochsri},  that the $1$-cycle
$\mathcal{C}_{1,x}-\mathcal{C}_{1,0_J}$ is rationally equivalent to $0$ in $X$.
This rational equivalence relation  is analyzed in \cite{voisinconiveau} and  this leads to the following statement, possibly after replacing $X$ by $X\times \mathbb{P}^r$ or a blow-up of $X$.
\begin{lemm} \label{lesecondstep}There exists  a family of stable maps
\begin{eqnarray} \label{eqstablemap2}   \begin{matrix}
 & \mathcal{C}_2 &\stackrel{f_2}{\rightarrow}& X
 \\
&p_2\downarrow& &
\\
&W_2& &
\end{matrix}
\end{eqnarray}
 parameterized by a Zariski open set $W_2\subset \mathbb{P}^1$ containing $0$ and $\infty$, satisfying  the following properties :
 \begin{enumerate}
  \item The fiber of $(p_2,f_2)$ over $0\in W_2$ is the union of  the fiber $f_{1,0_J}:\mathcal{C}_{1,0_J}\rightarrow X$ of the family (\ref{eqstablemap1}) over $0_J\in J(C)$ and  a stable map  $f_A:=f_{2\mid A}:A\rightarrow X$ glued at finitely many points. The fiber
 of $(p_2,f_2)$ over $\infty\in W_2$ is the union of  the fiber $f_{1,x}:\mathcal{C}_{1,x}\rightarrow X$ of the family (\ref{eqstablemap1}) over $x\in J(C)$ and  a stable map  $f_B:=f_{2\mid B}:B\rightarrow X$ glued at finitely many points.

 \item \label{refitem2} The two maps   $f_A: A\rightarrow X$ and $f_B:B \rightarrow X$ have the same image in $X$. More precisely, they are different partial normalizations of the same   nodal curve $E$ contained  in $X$.
     \end{enumerate}

\end{lemm}

{\bf Step 2.}    {\it Deforming $\mathcal{C}_{1,x}\cup B$ to $\mathcal{C}_{1,x}\cup A$}. The two curves or rather morphisms $$f_A:A\rightarrow X,\,f_B:B\rightarrow X$$ and their respective  attachments to the curves $\mathcal{C}_{1,0_J}$ and $\mathcal{C}_{1,x}$  are not equal as (pointed)  stable maps, although, by item \ref{refitem2} above,  they have the same support. They  may  differ combinatorially (for example the number of attachments points of $A$ and $B$  to the respective fibers  $\mathcal{C}_{1,0_J}$ and $\mathcal{C}_{1,x}$ may differ) and this  is why we cannot consider that the two fibers  $f_{2,0}: \mathcal{C}_{2,0}\rightarrow X$ and  $f_{2,\infty}: \mathcal{C}_{2,\infty}\rightarrow X$   are obtained by attaching the {\it same}  curve to the respective fibers  $f_{1,0_J}: \mathcal{C}_{1,0_J}\rightarrow X,\, f_{1,x}: \mathcal{C}_{1,x}\rightarrow X$ by the same number of points. The role of the family $\mathcal{C}_3\rightarrow W_3$  is  to correct this problem  by providing a deformation of  the curve $\mathcal{C}_{1,x}\cup B$ appearing in  the previous step to  a curve $\mathcal{C}_{1,x}\cup A$. The precise statement is      as follows.

\begin{lemm}   \label{lethirdstep}There exist  a family of stable maps
\begin{eqnarray} \label{eqstablemap3}   \begin{matrix}
 & \mathcal{C}_3 &\stackrel{f_3}{\rightarrow}& X
 \\
&p_3\downarrow& &
\\
&W_3& &
\end{matrix}
\end{eqnarray}
 parameterized by a Zariski open set $W_3$  in a smooth projective connected and simply connected variety, and two points $w_2,\,w_3\in W_3$  satisfying  the following properties :
 \begin{enumerate}
  \item   \label{curveR1} The fiber $f_{3,w_2}:\mathcal{C}_{3,w_2}\rightarrow X$ of the family (\ref{eqstablemap3})  over $w_2$ is the union of  the fiber $f_{2,\infty}:\mathcal{C}_{2,\infty}=\mathcal{C}_{1,x}\cup B\rightarrow X$ of the family (\ref{eqstablemap2}) over $\infty \in \mathbb{P}^1$ and  a curve $R\hookrightarrow X$, glued at finitely many points.

      \item \label{curveR} The fiber $f_{3,w_3}:\mathcal{C}_{3,w_3}\rightarrow X$
 of the family   (\ref{eqstablemap3}) over $w_3$ is the union of  the fiber $f_{1,x}:\mathcal{C}_{1,x}\rightarrow X$ of the family (\ref{eqstablemap1}) over $x\in J(C)$, the curve  $f_A:A\rightarrow X$ glued at finitely many points, and the same  curve $R\hookrightarrow X$ as in \ref{curveR1}, which is glued to the union $\mathcal{C}_{1,x}\cup A $  at finitely many points.

 \item\label{itemcondpoint1} The number of attachment points of $R\cup B$ to $\mathcal{C}_{1,x}$ in (\ref{curveR1}) equals the number of attachment points of $R\cup A$ to $\mathcal{C}_{1,x}$ in (\ref{curveR}).

 \item\label{itemcondpoint2} The number of attachment points of $R$ to $\mathcal{C}_{1,x}\cup B$ in (\ref{curveR1}) equals the number of attachment points of $R$ to $\mathcal{C}_{1,x}\cup A$ in (\ref{curveR}).

     \end{enumerate}

\end{lemm}

Note that we asked above  that $W_3$ is Zariski open in a  simply connected smooth projective variety. In \cite{voisinconiveau}, we just asked   the property that the Abel-Jacobi map of the family (\ref{eqstablemap3} ) is trivial and Lemma \ref{lethirdstep} was stated and proved  in  \cite{voisinconiveau}  in this slightly weaker form, but the same proof  shows in fact  the stronger statement.   Lemma \ref{lethirdstep} may seem technical but it is an immediate application of  the following simple fact.

\begin{sublemm}\label{sublemmhel} Let $\Gamma\subset S$ be a nodal curve in a smooth simply connected  surface. Let
$g_1:\Gamma_1\rightarrow S,\,g_2:\Gamma_2\rightarrow S$ be the two stable maps obtained by normalizing $\Gamma$ along a subset $N_1\subset \Gamma$ (resp.  $N_2\subset \Gamma$) of nodes, where $N_1$ and $N_2$ have the same cardinality. Then, up to replacing $\Gamma$ by a union  $\Gamma \cup R$, for some smooth curve $R\stackrel{i}{\hookrightarrow} S$ in general position, there
is a family of stable maps parameterized by a Zariski open set in a  connected simply connected projective variety, which has $g'_1=(g_1,i):\Gamma_1\cup R\rightarrow S$ and $g'_2=(g_2,i):\Gamma_2\cup R\rightarrow S$ as fibers.
\end{sublemm}
\begin{proof} We prove this result  by noting that the universal  family of curves in $S$ with $n$ specified  nodes  in the class $\Gamma\cup R$, for $R$ ample enough, is parameterized by a  projective  bundle $P$ over a  Zariski open set in  $S^{[n]}$. As $S$ is simply connected,  $S^{[n]}$  is simply connected, and $P$ admits a simply connected projective compactification. The desired family  is obtained by normalizing the curves  of  the universal family over $P$ at the specified set of $n$ nodes, which specializes to $N_1$ as set of specified  nodes for $\Gamma_1\cup R$ and to  $N_2$ as set of specified nodes for $\Gamma_2\cup R$.
\end{proof}
We deduce Lemma \ref{lethirdstep} from Sublemma \ref{sublemmhel} by choosing for $S$ a smooth surface complete intersection of ample hypersurfaces in $X$, containing the curves $\mathcal{C}_{1,x}$ and $E$. We may need to blow-up $X$ to guarantee that the union $\mathcal{C}_{1,x}\cup E$ is nodal in $X$ so that such a smooth surface $S$ exists. The simple connectedness of $S$ follows from that of $X$ which is rationally connected. Assume first that, in Lemma \ref{lesecondstep}, the number of attachments of $A$  to $\mathcal{C}_{1,0_J}$ and $B$ to $\mathcal{C}_{1,x}$  are the same. The curves  $\mathcal{C}_{1,0_J}$ and $\mathcal{C}_{1,x}$  have the same arithmetic genus since they are fibers of the family (\ref{eqstablemap1}), and the curves $\mathcal{C}_{2,0}=\mathcal{C}_{1,0_J}\cup A$ and $\mathcal{C}_{2,\infty}=\mathcal{C}_{1,x}\cup B$ have the same arithmetic genus since they are fibers of the family (\ref{eqstablemap2}). It follows  that, in that case,  the two partial  normalization maps
$f_A:A\rightarrow E$ and $f_B:B\rightarrow E$ normalize the same number of points of $E$, so that Sublemma \ref{sublemmhel} applies directly to $A=\Gamma_1,\,B=\Gamma_2,\,E=\Gamma$. In general, we reduce to the previous case by gluing an extra component to both $\mathcal{C}_{2,0}$ and $\mathcal{C}_{2,\infty}$ and by  modifying in an adequate way the family $\mathcal{C}_2$ of (\ref{eqstablemap2}) using Lemma \ref{ledattachement} below.

\vspace{0.5cm}

{\bf Step 3.} {\it Attaching curves to glue the families.}  Although the fact that the stable maps  $A$ and $B$ had  the same image in $X$  played an important role in the previous proof, we  consider them as distinct curves inside $X$, with empty  intersection, which can be done by replacing $X$ by $X\times \mathbb{P}^1$. All the curves considered below are  curves in $X$ and the unions we write  are unions of  curves in $X$ with transverse intersections.

We now have three families
$\mathcal{C}_1,\,\mathcal{C}_2,\, \mathcal{C}_3$  respectively  over $W_1$, $W_2$ and $W_3$,  with

-  respective special fibers $\mathcal{C}_{1,0_J},\,\mathcal{C}_{1,x}$  over $0_J,\,x\in W_1$,

-  respective special fibers $\mathcal{C}_{1,0_J}\cup A,\,\mathcal{C}_{1,x}\cup B$  over $0,\,\infty\in W_2$, and

-  respective special fibers $\mathcal{C}_{1,x}\cup B \cup R,\, \mathcal{C}_{1,x}\cup A\cup R$ over   $w_2,\,w_3\in W_3$.

 In other words, if we make correspond $0$ with  $0_J$, $\infty$ with  $w_2$, and $w_3$ with $x$,  the corresponding special fibers of the  three families    coincide after attaching  $A\cup R$ to the fibers of $f_1$,   and  $R$ to the fibers of $f_2$.

We now explain how we can  modify the  families above in order to make them  actually having the same corresponding  special fibers.  A key point is the following result which allows to attach curves in families, without changing the Abel-Jacobi map.
\begin{lemm}   \label{ledattachement}  Let $Y$ be a smooth projective  variety  and let
$$p:\mathcal{C}\rightarrow B,\,\, \phi: \mathcal{C}\rightarrow Y$$ be a family of stable maps with value in $Y$, that we assume to be injective for simplicity,  parameterized by a connected basis $B$,  let $T\subset Y$ be a curve intersecting transversally two fibers $\mathcal{C}_{b_1},\,\mathcal{C}_{b_2}$ along   sets  $N_1:=\mathcal{C}_{b_1} \cap T$, resp. $N_2:=\mathcal{C}_{b_2}\cap T$ of  smooth points. Then for an adequate choice of curve $D\subset Y$,  there exists
a family
 $\phi': \mathcal{C}'\rightarrow Y$, $p':\mathcal{C}'\rightarrow B^0$, where $B^0$ is a Zariski open set of $B$ containing $b_1$ and $b_2$, with the following properties.
 \begin{enumerate}
 \item The fibers of  the family $\mathcal{C}'$ are obtained by attaching  curves to the fibers of the family $\mathcal{C}$.
 \item The attached curves have constant rational equivalence class. In particular, the  Abel-Jacobi maps of the families  $\mathcal{C}'$ and  $\mathcal{C}$ are equal.
 \item Over $b_1$, resp.  $b_2$, the fiber $\mathcal{C}'_{b_1}$, resp.  $\mathcal{C}'_{b_2}$, is the transverse   union of curves in $Y$
 $\mathcal{C}_{b_1}\cup T\cup D$, resp.  $\mathcal{C}_{b_2}\cup  T\cup D$.
 \end{enumerate}
 \end{lemm}

A key point for us is the fact that the same curve $D$ can be used for different families, as it is here just to  give more positivity.
\begin{proof}[Proof of Lemma \ref{ledattachement}] For simplicity, we prove the statement when $Y$ is a threefold. Let $S\subset Y$ be a smooth projective surface containing $T$ and in general position otherwise. We write  $$\mathcal{C}_{b_1}\cap S=N_1\cup N'_1,\,\,\mathcal{C}_{b_2}\cap S=N_2\cup N'_2,$$
and we choose a very ample linear system of curves on $S$, which contains
a curve of the form $T\cup D$, intersecting $\mathcal{C}_{b_1}$ along $N_1\cup N'_1$, and $\mathcal{C}_{b_2}$ along $N_2\cup N'_2$.
The general curve $\mathcal{C}_b$, $b\in B$, intersects $S$ along a set $N''_b$ specializing to   $N_1\cup N'_1$ and $N_2\cup N'_2$, and, as the linear system $|T+D|$ is chosen ample enough on $S$, there is a projective bundle over $B$ parameterizing curves $D'_b$ in $|T+D|$ passing through $N''_b$.
We choose a rational section $b\mapsto D'_b$ of this projective bundle which is well-defined on a Zariski open set of $B$ containing $b_1$ and $b_2$ and passes through the point parameterizing  $T\cup D$ at both points $b_1$, $b_2$, and attach
$D'_b$ to $\mathcal{C}_b$ along $N''_b$. As all the glued components belong to the same linear system on $S$, they all have the same rational equivalence class in $Y$.
\end{proof}

We apply  Lemma \ref{ledattachement}  to the three families already constructed. The curve $R$ being as in \ref{curveR} of Lemma \ref{lethirdstep}, this lemma provides us with
a family of stable maps
$$p'_1: \mathcal{C}'_1\rightarrow W_1,\,f '_1: \mathcal{C}'_1\rightarrow X$$
parameterized by $W_1$  (or a Zariski dense open set of it containing its two special fibers), with special fibers
$\mathcal{C}_{1,0_J}\cup A\cup R\cup  D_1$ over $0_J$ and  $\mathcal{C}_{1,x}\cup A\cup R\cup D_1$ over $x$.

Similarly,  Lemma \ref{ledattachement} provides us with
a family of stable maps
$$p'_2: \mathcal{C}'_2\rightarrow W_2,\,f '_2: \mathcal{C}'_1\rightarrow X$$
parameterized by $W_2$  (or a Zariski dense open set of it containing its two special fibers), with special fibers
$\mathcal{C}_{2,0}\cup R\cup D_2=\mathcal{C}_{1,0_J}\cup A\cup R\cup D_2$ over $0\in \mathbb{P}^1$ and  $\mathcal{C}_{2,\infty}\cup R\cup D_2=\mathcal{C}_{1,x}\cup B \cup R \cup D_2$ over $x$.

Finally,  Lemma \ref{ledattachement} (with an empty curve) provides us with
a family of stable maps
$$p'_3: \mathcal{C}'_3\rightarrow W_3,\,f '_3: \mathcal{C}'_3\rightarrow X$$
parameterized by $W_3$  (or a Zariski dense open set of it containing its two special fibers), with special fibers
$\mathcal{C}_{3,w_2}\cup D_3=\mathcal{C}_{1,x}\cup B\cup R\cup D_3$ over $w_2\in  W_3$ and  $\mathcal{C}_{3,w_3} \cup D_3=\mathcal{C}_{1,x}\cup A \cup R \cup D_3$ over $w_3\in W_3$.

There is a lot of freedom in choosing the  curves $D_i$, as shows the proof of Lemma \ref{ledattachement}.  Looking at this proof and using the numerical conditions (\ref{itemcondpoint1}) and (\ref{itemcondpoint2}), we check that we can  choose  $D_1=D_2=D_3$. The three families $\mathcal{C}'_1$,   $\mathcal{C}'_2$,  $\mathcal{C}'_3$ then have identical corresponding special fibers, hence they glue to give a family of stable maps
parameterized by the union
$$B=W_{1}\cup_{0_J=0,\,x=w_3}(W_2\cup_{\infty=w_2} W_3).$$
Theorem \ref{theoauxi} is then almost proved, except for property  \ref{4}  (unobstructedness), which we will achieve now by a further modification of the family.

\vspace{0.5cm}

{\bf Step 4.} {\it   Realizing unobstructedness. }  Unobstructedness is obtained by applying the following
\begin{prop}\label{propunobstructed} Let $X$ be a rationally connected smooth projective variety and let
\begin{eqnarray}  \label{eqcorrespRC} \begin{matrix}
 & \mathcal{C}&\stackrel{f}{\rightarrow}& X
 \\
&p\downarrow& &
\\
&W& &
\end{matrix}
\end{eqnarray}
be a family of stable maps to $X$ parameterized by a quasiprojective variety $W$. Let $w_1,\ldots,\,w_N$ be points of $W$. Then there exists a family of stable maps
\begin{eqnarray}  \label{eqstablemodifiedprime} \begin{matrix}
 & \mathcal{C}'&\stackrel{f'}{\rightarrow}& X
 \\
&p'\downarrow& &
\\
&W^0& &
\end{matrix}
\end{eqnarray}
parameterized by   a Zariski open set $W^0$  of $W$ containing the points $w_i$, with the following properties.

\begin{enumerate}
\item   \label{item1pouruno} For any $t\in W^0$, the stable map $f'_t:\mathcal{C}'_t\rightarrow X$ is unobstructed.
\item \label{item2pouruno}   For any $t\in W^0$, the stable map $f'_t:\mathcal{C}'_t\rightarrow X$ is of the form
$$  \mathcal{C}'_t=\mathcal{C}_t\cup   \mathcal{E}_t,\,\,f'_{t\mid \mathcal{C}_t}=f_t, $$
and  the $1$-cycles $f'_{t*}\mathcal{E}_t$ are all rationally equivalent in $X$.
\end{enumerate}
\end{prop}

\begin{proof}  This proposition is essentially proved in \cite{voisinconiveau} although maybe not stated with this degree of generality, so we just sketch the proof.
First of all,  we can assume  without loss of generality  (we can indeed replace $X$ by $X\times \mathbb{P}^r$,  which has  isomorphic  group ${\rm CH}_{1,{\rm hom}}$ and intermediate Jacobian $J_3$)  that the maps $f_t$ are generically  injective and attach to the generic curve  $\mathcal{C}_t$ some complete intersections $C_1,\ldots, C_M$  of ample hypersurfaces in $X$ at the  $M$ points $x_1,\ldots, x_m$  in  the intersection $f_{t}(\mathcal{C}_t)\cap Y$, where $Y$ is a sufficiently ample hypersurface in $X$. These complete intersection curves have constant rational equivalence class.  It is proved by Hartshorne and Hirschowitz in \cite{hahi} that this operation, which can be done over a Zariski open set of $W$ containing the $N$ given points,  provides for $M$ large enough and  for generic choices of attaching directions,  local complete intersection  curves $f''_t:\mathcal{C}''_t\hookrightarrow X$,  $\mathcal{C}''_t=\mathcal{C}_{t}\cup C_1\cup\ldots\cup C_M$, with the property that the normal bundle $N_{\mathcal{C}''_{t}/X}$ restricted to $\mathcal{C}_t$ is arbitrarily positive, hence satisfies  $H^1(\mathcal{C}_{t},N_{\mathcal{C}''_{t}/X\mid \mathcal{C}_t})=0$. Unfortunately, this is not enough to guarantee that the curves $\mathcal{C}''_{t}$ are unobstructed, since we need for this the stronger property that the normal   bundle $N_{\mathcal{C}''_{t}/X}$  has  $H^1(\mathcal{C}''_{t},N_{\mathcal{C}''_{t}/X})=0$. This stronger condition is achieved by gluing as in \cite{komimo} very free rational curves $R_{ij}$  to each of the  complete intersection components $C_i$.
We can perform first   this  second construction over the function field of the variety parameterizing curves   complete intersections of very ample hypersurfaces of given degree. This  variety is rational so the union
$\cup_jR_{ij}$ of the  rational legs we add to $C_i$ is a $1$-cycle of  constant rational equivalent class in $X$. By a slight modification of   \cite{komimo}, the curve
$$\mathcal{C}'_t=\mathcal{C}''_t\cup_{ij} R_{ij}=\mathcal{C}_t\cup C_1\cup\ldots\cup C_m\cup_{ij} R_{ij}$$
is unobstructed if we add enough free rational curves in general position, so \ref{item1pouruno} is satisfied. By construction, this curve also satisfies   \ref{item2pouruno}, proving the proposition.
\end{proof}
By applying Proposition \ref{propunobstructed} to the family constructed in Step 3, we finally get a family satisfying all the properties stated in Theorem \ref{theoauxi}, which is thus proved.
\section{The case of $1$-cycles on higher dimensional rationally connected manifolds \label{sectheowalker}}
We now turn to the case of $1$-cycles on higher dimensional rationally connected manifolds. There are two important differences with the case of dimension $3$, as  Theorem \ref{blochsri} is not true for two reasons. First of all, as mentioned in the introduction, and can be   seen when we blow-up a surface inside a fourfold,  the ${\rm CH}_0$ group  of a surface can be present  in ${\rm CH}_1(X)$, and by \cite{mumford}, we do not expect a Mumford representability result  for ${\rm CH}_1(X)$, which can be infinite dimensional even when $X$ is rational. In particular,  the Bloch-Srinivas  injectivity theorem (Theorem \ref{blochsri})  for the Abel-Jacobi map on  codimension $2$ cycles  algebraically equivalent to $0$ is not true for $1$-cycles on rationally connected $n$-folds. This was our starting point in the proof of Theorem \ref{theoauxi}.  A second phenomenon is the fact that the Abel-Jacobi map
$$\Phi_X:{\rm CH}^{n-1}(X)\rightarrow J^{2n-3}(X)=J_3(X)$$
for $1$-cycles on rationally connected manifolds
 is a priori  not universal (in the sense of \cite{murre})  anymore, as it  admits the Walker factorization  \cite{walker}, which one can describe as follows: Recall that, inside $H^{2n-3}(X,\mathbb{Z})_{\rm tf}$, we have the subgroup
 $N^{n-2}H^{2n-3}(X,\mathbb{Z})_{\rm tf}$ of homology classes supported on a possibly singular surface $S\subset X$.
 For a rationally connected $n$-fold, one has
 $$H^{2n-3}(X,\mathbb{Q})=N^{n-2}H^{2n-3}(X,\mathbb{Q})$$ as shows an argument involving the decomposition of the diagonal with $\mathbb{Q}$-coefficients \cite{blochsrinivas}, and it follows that
 $$N^{n-2}H^{2n-3}(X,\mathbb{Z})_{\rm tf}\subset H^{2n-3}(X,\mathbb{Z})_{\rm tf}$$
 is a finite index subgroup. This determines a modified intermediate Jacobian
 $$J^{2n-3}(X)_{\rm Walker}:=H^{2n-3}(X,\mathbb{C})/(F^{n-1}H^{2n-3}(X)\oplus N^{n-2}H^{2n-3}(X,\mathbb{Z})_{\rm tf}),$$
 which is obviously isogenous to $J^{2n-3}(X)$,
 and it is proved in \cite{walker} that  the Abel-Jacobi map $\Phi_X^{n-1}$ factors through
 $$\Phi_{X,{\rm Walker}}^{n-1}: {\rm CH}^{n-1}(X)_{\rm alg}\rightarrow J^{2n-3}(X)_{\rm Walker}.$$
 Note that a  factorization of this type  holds for the Abel-Jacobi map for cycles on any codimension on any smooth projective variety (see \cite{walker}).

  As mentioned above, Theorem \ref{blochsri} is not true anymore for $1$-cycles  on rationally connected manifolds for $n\geq 4$. However,  one has the following result for torsion $1$-cycles, which we  will use for the proof of Theorem \ref{theoWalker}.
 \begin{theo} \label{theosuzuki} (Suzuki \cite{suzuki}.) The Walker Abel-Jacobi map $\Phi_{X,{\rm Walker}}^{n-1}$ is injective on the torsion subgroup of $1$-cycles algebraically equivalent to zero modulo rational equivalence of a rationally connected $n$-fold.
 \end{theo}
 We now prove Theorem \ref{theoWalker}, which is the optimal higher dimensional generalization of Theorem \ref{theomain}.
\begin{proof}[Proof of Theorem \ref{theoWalker}]  As explained in  Remark \ref{remapbplace}, there exists an abelian variety $J$ with a  codimension $n-1$ cycle
$Z\in {\rm CH}^{n-1}(J\times X)$ such that the induced Walker Abel-Jacobi map
$$\Phi_{Z,{\rm Walker}}: {\rm Alb}(J)=J\rightarrow  J^{2n-3}(X)_{\rm Walker}$$  is an isogeny.  The kernel of $\Phi_{Z,{\rm Walker}}$ thus consists of torsion points of $J$.
We now use the following
\begin{lemm} For any abelian variety $A$ and torsion point $x\in A$, the zero-cycle
$\{x\}-\{0_A\}\in{\rm CH}_0(A)$ is of  torsion.
\end{lemm}
It follows that for any $x\in {\rm Ker}\,\Phi_{Z,{\rm Walker}}$, the cycle
$Z_x-Z_{0_J}=Z_*(\{x\}-\{0_J\})$ is a torsion element of ${\rm CH}_1(X)_{\rm alg}$. Moreover, as $x\in {\rm Ker}\,\Phi_{Z,{\rm Walker}}$, this cycle is annihilated by $\Phi_{X,{\rm Walker}}^{n-1}$. It then follows from Theorem \ref{theosuzuki} that
\begin{eqnarray} \label{eqvanZsuzuki} Z_x-Z_{0_J}=0\,\,{\rm in}\,\,{\rm CH}_1(X).
\end{eqnarray}
  We can thus apply Theorem \ref{theoauxi}  which only uses the vanishing (\ref{eqvanZsuzuki}) and rational connectedness of $X$. We get in turn   the following
  version of  Theorem \ref{theonewform}  which we obtain  as  in Section \ref{secprooftheonewform}  as a    consequence of Theorem  \ref{theoauxi}.
  \begin{prop}\label{propversionduthm} There exist a smooth projective variety $M_0$,  a codimension $n-1$ cycle
  $$Z_{M_0}\in{\rm CH}^{n-1}(M\times X)$$
  and a rational map $\phi_0: J\dashrightarrow M_0$, with the following properties
  \begin{enumerate}
  \item \label{item1suzuji} $ \Phi_{Z_{M_0},{\rm Walker}}\circ \phi_*= \Phi_{Z,{\rm Walker}}: {\rm Alb}(J)=J\rightarrow  J^{2n-3}(X)_{\rm Walker}$.
  \item \label{item2suzuki}  $x$ belongs to the kernel of the  morphism  $\phi_{0*}:{\rm Alb}(J)\rightarrow {\rm Alb}(M_0)$ induced by $\phi_0$.
  \end{enumerate}
  \end{prop}
   Note that, by item \ref{item1suzuji}, the kernel of the  morphism  $\phi_{0*}:{\rm Alb}(J)\rightarrow {\rm Alb}(M_0)$ is contained in   ${\rm Ker}\,\Phi_{Z,{\rm Walker}}$.

   We now iterate this process. For any $y\not=x \in {\rm Ker}\, \Phi_{Z,{\rm Walker}}^{n-1}$, we consider the point $y'=\phi_*y\in M_0$. By  item \ref{item1suzuji}, we have
   $\Phi_{Z_{M_0},{\rm Walker}}(y'-0')=0$, that is $\Phi_{X,{\rm Walker}}^{n-1}(Z_{M_0,y'}-Z_{M_0,0'})=0$ in   $ J^{2n-3}(X)_{\rm Walker}$, where $0':=\phi_0(0_J)$. As the zero-cycle $y'-0'$ is of torsion in $M_0$, the cycle
  $ Z_{M_0,y'}-Z_{M_0,0'}$ is of torsion in ${\rm CH}_1(X)_{\rm alg}$, hence we conclude as before by Theorem \ref{theosuzuki}  that $ Z_{M_0,y'}-Z_{M_0,0'}=0$ in  ${\rm CH}_1(X)_{\rm alg}$. We then construct as above a smooth projective variety $M_1$, a rational map $\phi':M_0\dashrightarrow M_1$ and a codimension $n-1$ cycle
  $$Z_{M_1}\in {\rm CH}^{n-1}(M_1\times X)$$ such that

   \begin{enumerate}
  \item \label{item1suzujiprime} $ \Phi_{Z_{M_1},{\rm Walker}}\circ \phi'_*= \Phi_{Z_{M_0},{\rm Walker}}: {\rm Alb}(M_0)\rightarrow  J^{2n-3}(X)_{\rm Walker}$.
  \item \label{item2suzukiprile}  ${\rm alb}(y'-0')$ belongs to the kernel of the  morphism  $\phi'_*:{\rm Alb}(M_0)\rightarrow {\rm Alb}(M_1)$ induced by $\phi'$.
  \end{enumerate}
  The composition $\phi_1=\phi'\circ \phi_0: J\dashrightarrow M_1$  has now the property that
   \begin{enumerate}
  \item \label{item1suzujiter} $ \Phi_{Z_{M_1},{\rm Walker}}\circ \phi_{1}*= \Phi_{Z,{\rm Walker}}: {\rm Alb}(J)=J\rightarrow  J^{2n-3}(X)_{\rm Walker}$.
  \item \label{item2suzukiter}  $ y$ and $x$  belong  to the kernel of the  morphism  $\phi_{1}*:{\rm Alb}(J)=J\rightarrow {\rm Alb}(M_1)$ induced by $\psi$.
  \end{enumerate}
  As the kernel  ${\rm Ker}\, \Phi_{Z,{\rm Walker}}$ is finite, we finally get by iterating this process a smooth projective variety $N$,  a cycle
  $Z_{N}\in {\rm CH}^{n-1}(N\times X)$, and a rational map $\psi: J\dashrightarrow N$ satisfying the properties

   \begin{enumerate}
  \item \label{item1suzujifour} $ \Phi_{Z_{N},{\rm Walker}}\circ \psi_{*}= \Phi_{Z,{\rm Walker}}: {\rm Alb}(J)=J\rightarrow  J^{2n-3}(X)_{\rm Walker}$.
  \item \label{item2suzukifour} $ {\rm Ker}\, \Phi_{Z,{\rm Walker}}$ is contained in
  ${\rm Ker}\,( \psi_{*}:{\rm Alb}(J)\rightarrow  {\rm Alb}(N))$.
  \end{enumerate}
  The conditions \ref{item1suzujifour} and \ref{item2suzukifour} imply that    the abelian subvariety ${\rm Im}\,\psi_{*}\subset {\rm Alb}(N)$ is isomorphic to $J/{\rm Ker}\,\Phi_{Z,{\rm Walker}}$,  which is  isomorphic to $J^{2n-3}(X)_{\rm Walker}$ via $\Phi_{Z_{N},{\rm Walker}}$.

  The proof is then concluded by applying Lemma \ref{levenantdeconiveau}, which gives us a smooth projective variety $M$ together with a morphism $f: M\rightarrow N$, such that
  $f_*:{\rm Alb}(M)\rightarrow   {\rm Alb}(N)$ is an isomorphism onto  ${\rm Im}\,\psi_{*}\cong  J^{2n-3}(X)_{\rm Walker}$, and by considering the restricted cycle  $(f,Id_X)^*(Z_{N})\in {\rm CH}^{n-1}(M\times X)$.
\end{proof}

\section{Further results and questions}
\subsection{Stronger representability results\label{sectionopen}}
Recall from \cite{voisinJAG}, \cite{voisinuniv} that a codimension $2$-cycle
 $Z_{J}\in {\rm CH}^2(J^3(X)\times X)$ is  a universal codimension $2$ cycle if it satisfies
 $$\Phi_{Z_J}=Id_{J^3(X)}: {\rm Alb}(J^3(X))=J^3(X)\rightarrow J^3(X).$$
 Similarly, for any smooth projective variety $M$ of dimension $n$, a universal $0$-cycle
 parameterized by $A:={\rm Alb}(M)$ is a codimension $n$ cycle
 $Z_A\in{\rm CH}^n(A\times M)$
  such that
 $${\rm alb}_M\circ Z_{A*}: {\rm Alb}(A)= {\rm Alb}(M)\rightarrow {\rm Alb}(M)$$
 is the identity.
 We observed in \cite{voisinJAG} that  the existence of  a universal codimension $2$-cycle on $X$ is a necessary criterion for the existence of a cohomological decomposition of the diagonal of $X$, and  hence for the   stable rationality of $X$. We proved in \cite{voisinuniv} that some rationally connected threefolds do not admit a universal codimension $2$ cycle (this is known to  happen for a very general quartic double solid with seven nodes by \cite{voisinuniv} and it follows that this also holds for a very general quartic double solid). We get the following consequence
of  Theorem \ref{theomain}.
\begin{coro} \label{coroplustard}  Let $X$ be a rationally connected threefold and let $M,\,Z_M$ be as in Theorem \ref{theomain}. Then $X$ admits a universal codimension $2$ cycle if  $M$ admits a universal $0$-cycle.

In particular, for a very general quartic double solid $X$ and any smooth projective variety $M$ equipped with a codimension $2$ cycle $Z_M\in{\rm CH}^2(M\times X)$ such that $\Phi_{Z_M}:{\rm Alb}(M)\rightarrow J^3(X)$ is an isomorphism, $M$ does not admit a universal $0$-cycle.
\end{coro}
\begin{rema}{\rm Whether the converse  holds for an adequate choice of pair $(M,Z_M)$  is related to the question whether we can choose  $(M,Z_M)$ universally generating, that we are going to discuss below.}
\end{rema}
\begin{proof}[Proof of Corollary \ref{coroplustard}] The proof is immediate  since a universal $0$-cycle $$Z_{{\rm univ},M}\in{\rm CH}^{{\rm dim}\,M}({\rm Alb}(M)\times M)$$ gives rise to a codimension $2$-cycle
$$Z_{{\rm univ},X}:=Z_M\circ Z_{{\rm univ},M}\in{\rm CH}^{2}({\rm Alb}(M)\times X)$$
which is universal for $X$ using the fact that $\Phi_{Z_M}$ induces an isomorphism between
${\rm Alb}(M)$ and $J^3(X)$.
\end{proof}

We  now prove  Corollary \ref{corosursdir} stated in the introduction.
\begin{proof}[Proof of Corollary  \ref{corosursdir}]   Let $M_0$ be a smooth projective variety not admitting a universal $0$-cycle. Consider the variety
$$M=M_0\times {\rm Alb}(M_0).$$
As ${\rm Alb}(M)={\rm Alb}(M_0)\oplus {\rm Alb}(M_0)$, it contains the abelian subvariety
$\Delta={\rm Diag}({\rm Alb}(M_0))\subset  {\rm Alb}(M_0)\oplus {\rm Alb}(M_0)$ which is obviously a direct summand of ${\rm Alb}(M)$  since
$$ ({\rm Alb}(M_0)\times \{0\})\oplus  \Delta= {\rm Alb}(M).$$
We claim that there do not exist  a smooth projective variety $N$ and a pair of correspondences
$\Gamma_1\in{\rm CH}^d(N\times M)$, $\Gamma_2\in {\rm CH}^{d'}(M\times N)$, where $d={\rm dim}\,M$, $d'={\rm dim}\,N$,  such that
$$\Gamma_{2*}\circ  \Gamma_{1*}=Id_{{\rm Alb}(N)}$$
and ${\rm Im}\,(\Gamma_{1*}: {\rm Alb}(N)\rightarrow {\rm Alb}(M))=\Delta$.  Indeed, denote by $i$ the inclusion ${\rm Alb}(M_0)= \{m_0\}\times  {\rm Alb}(M_0)\hookrightarrow M$. Assume $N,\,\Gamma_1,\,\Gamma_2$ as above exist. Then
consider the correspondence
$$ \Gamma:=pr_1\circ \Gamma_1\circ \Gamma_2\circ i\in {\rm CH}^d({\rm Alb}(M_0)\times M_0),\,\,d={\rm dim}\,M_0.$$
Then  $\Gamma_*:{\rm Alb}\,M_0\rightarrow  {\rm Alb}\,M_0$ is the identity, which provides a universal $0$-cycle for $M_0$ and is a contradiction.
\end{proof}

To compare various notions of representability for $J^3(X)$, note  that, if $X$ admits a universal codimension $2$ cycle, then we can take for $M$ in Theorem \ref{theomain}  the intermediate Jacobian $J^3(X)$ itself, and for $Z_M$ any universal codimension $2$ cycle. Thus the existence of a universal codimension $2$ cycle is a stronger statement than Theorem \ref{theomain}, but it is too strong and not satisfied in general.
This  shows that Theorem \ref{theomain} is almost optimal.
However, it
 still leaves open  a slightly stronger statement, coming from  Shen's notion of universal generation \cite{shen}, namely
 \begin{Defi} \label{defimshen} (M. Shen \cite{shen}.) A codimension $k$ cycle $Z_M\in{\rm CH}^k(M\times X)$, where $M$ is smooth projective of dimension $d$,  is universally generating if it  has the property that for any smooth projective variety $Y$ and codimension $k$ cycle
 $Z_Y\in{\rm CH}^k(Y\times X)$, there exists a codimension $d$ cycle $W_Y\in {\rm CH}^d(M\times Y)$ such that
 $$ Z_{Y*}=Z_{M*}\circ W_{Y*}:{\rm CH}_0(Y)_{\rm hom}\rightarrow {\rm CH}^k(X)_{\rm alg}.$$
 \end{Defi}
 We now consider  the case where $k=2$ and $X$ is rationally connected, so that  the Abel-Jacobi map
 $\Phi_{X}: {\rm CH}^2(X)_{\rm alg}\rightarrow J^3(X)$
 is an isomorphism by Theorem \ref{blochsri}.  An easy consequence of Theorem \ref{theomain} is the following result.
 \begin{coro}\label{letrivialdirfac} Let $X$ be a rationally connected threefold. If $X$ admits  a universally generating family $(M,Z_M)$, then $J^3(X)$ is  a direct summand of ${\rm Alb}(M)$.
 \end{coro}
 \begin{proof} We know by Theorem \ref{theomain} that there also exist  a smooth projective variety $N$ and a  codimension $2$ cycle $Z_N\in{\rm CH}^2(N\times X)$ such that
 $\Phi_{Z_N}: {\rm Alb}(N)\rightarrow J^3(X)$ is an isomorphism. By the universal generation property, there exists a cycle $\Gamma\in{\rm CH}^d(N\times M)$, $d={\rm dim}\,M$, such that
 $$\Phi_{Z_N}=\Phi_{Z_M}\circ \Gamma_*:{\rm Alb}(N)\rightarrow  J^3(X).$$
 It thus follows that $\Gamma_*({\rm Alb}(N))\subset  {\rm Alb}(M)$ is isomorphic to $J^3(X)$ and $\Phi_{Z_M}$ then provides a retraction of   ${\rm Alb}(M)$  onto  $J^3(X)$, making it into  a direct summand.
 \end{proof}
In fact,  we can prove a better statement:
\begin{theo}  \label{theoshenunivalbiso} Assume a rationally connected threefold $X$ admits a universally generating family $(N,Z_N)$. Then $X$ admits   a universally generating family $(M,Z_M)$ such that
$\Phi_{Z_M}: {\rm Alb}(M)\rightarrow J^3(X)$ is an isomorphism.
\end{theo}
\begin{proof} The proof will be done in several steps.

{\bf Step 1}. We first prove by mimicking the proof of Theorem \ref{theomain} the following
\begin{prop} \label{prop1}  Let $(N,Z_N)$ be a pair consisting of a  smooth projective variety  $N$ and a codimension $2$ cycle
$Z_N\in{\rm CH}^2(N\times X)$.  Then there exist a smooth projective variety $N'$, a codimension $2$ cycle $Z_{N'}\in {\rm CH}^2(N'\times X)$, and a rational map
$$\phi:N\dashrightarrow N'$$
with the following properties.

(1) $\Phi_{Z_N}=\Phi_{Z_{N'}}\circ \phi: N\dashrightarrow J^3(X)$.

(2) ${\rm Ker}\,(\Phi_{Z_{N}}:{\rm Alb}(N)\rightarrow J^3(X))\subset  {\rm Ker}\,(\phi_{*}:{\rm Alb}(N)\rightarrow {\rm Alb}(N'))$.
\end{prop}
Note  that (1) and (2) in fact  imply that
$${\rm Ker}\,(\Phi_{Z_{N}}:{\rm Alb}(N)\rightarrow J^3(X))=  {\rm Ker}\,(\phi_{*}:{\rm Alb}(N)\rightarrow {\rm Alb}(N')).$$

{\bf Step 2}. Assume now that $(N,Z_N)$ is universally generating. Then for any smooth projective variety $Y$ and correspondence $\Gamma\in{\rm CH}^2(Y\times X)$, there exists a correspondence
$\Gamma'\in {\rm CH}^{n}(Y\times N)$, $n={\rm dim}\,N$, such that
$$ \Phi_\Gamma=\Phi_{Z_N}\circ \Gamma'_*:{\rm Alb}(Y)\rightarrow J^3(X).$$
Let $\Gamma'':=\phi\circ \Gamma' \in {\rm CH}^{n'}(Y\times N')$,  $n'={\rm dim}\,N'$, where $N',\,\phi$ are as in Step 1. Then by (1) above, we have
$$\Phi_\Gamma=\Phi_{Z_N}\circ \Gamma'_*=\Phi_{Z_{N'}}\circ\phi_*\circ \Gamma'_*=\Phi_{Z_{N'}}\circ \Gamma''_*:{\rm Alb}(Y)\rightarrow J^3(X).$$
It follows that $(N',Z_{N'})$ is also  universally generating. Unfortunately, we can have ${\rm Alb}(N')\not=J^3(X)$ so we have to modify the construction in an adequate way. We observe that, as proved in Corollary \ref{letrivialdirfac}, the image
of $\phi_*: {\rm Alb}(N)\rightarrow  {\rm Alb}(N')$ is a direct summand $A$ of $ {\rm Alb}(N')$ which is isomorphic to $J^3(X)$.
By definition of $A$, the composite
${\rm alb}_{N'}\circ \phi:N\rightarrow   {\rm Alb}(N')$ takes value in $A$. Our next proposition is the following complement to Lemma \ref{levenantdeconiveau}.
\begin{prop} \label{pro2}   Let $N,\,N'$ be two smooth projective varieties and $\phi:N\rightarrow N'$ be a morphism (or rational map).
Let $A:={\rm Im}\,(\phi_*:{\rm Alb}(N)\rightarrow  {\rm Alb}(N'))\subset  {\rm Alb}(N')$. Then there exist  a smooth projective variety $N''$, a rational map
$\psi: N\dashrightarrow N''$, and  a correspondence
$\gamma\in {\rm CH}^{n'}(N''\times N')$ with the following properties:

(3) $\gamma_*\circ \psi_*=\phi_*:   {\rm Alb}(N)\rightarrow  {\rm Alb}(N')$.

(4)  $\gamma_*:{\rm Alb}(N'')\rightarrow  {\rm Alb}(N')$ is injective with image $A$.

\end{prop}
Let now $M=N'', \,Z_M:=  Z_{N'}\circ \gamma\in{\rm CH}^2(M\times X)$. By Proposition \ref{prop1} (1)  and Proposition \ref{pro2} (3), we get by the same argument as above that  $(M,Z_M)$ is universally generating. Using  Proposition \ref{pro2} (4), we find that $\Phi_{Z_M}:{\rm Alb}(M)\rightarrow J^3(X)$ is injective, hence an isomorphism. This completes the proof of Theorem \ref{theoshenunivalbiso} admitting Proposition   \ref{pro2}.
\end{proof}
\begin{proof}[Proof of Proposition \ref{pro2}]  Let $n'={\rm dim}\,N'$ and let $L$ be  a very ample line bundle  on $N'$. A choice of $n'$ general sections of $L$ gives us  a dominant rational map
$f:N'\dashrightarrow\mathbb{P}^{n'-1}$ whose general fiber $C_t$, $t\in \mathbb{P}^{n'-1}$,  is a smooth complete intersection of $n'-1$ hypersurfaces in $|L|$.  Let $r$ be any integer and  let   $P_{r}$ be  any smooth projective model of the relative Picard variety ${\rm Pic}^{r} (N'/\mathbb{P}^{n'-1})$. Let $x_0\in N'$ be one of the base points of the linear system. We thus have   a rational map $a_{r}:  P_{r}\dashrightarrow {\rm Alb}(N')$ (which in fact is a morphism)
given on the general point  $d_t\in  {\rm Pic}^{r} (C_t)$ by
$$ a_{r}(d_t)= {\rm alb}_{N'}(j_{t*}(d_t-rx_{0,t})),$$
where $j_t$ is the inclusion of $C_t$ in $N'$ and $x_{0,t}$ is the point $x_0$ seen as a point of $C_t$.

Let now $N'':=a_{r}^{-1} (A_s)$, where $A_s\subset {\rm Alb}(N')$ is a general translate of $A$.  We first claim that
${\rm Alb}(N'')=A$. This follows indeed from the fact that the general fiber of  $a_{r}$ is a fibration over $ \mathbb{P}^{n'-1}$ with fiber the abelian variety ${\rm Ker}\,(j_{t*}:J(C_t)\rightarrow {\rm Alb}(N'))$.
Indeed, by Lefschetz theorem on hyperplane sections, ${\rm alb}_{N'}\circ j_{t_*}: J(C_t)\rightarrow {\rm Alb}(N')$ is surjective (with connected fibers) with  kernel the abelian variety built on  the weight $1$ Hodge structure
$H_1(C_t,\mathbb{Z})_{\rm van}:={\rm Ker}\,(j_{t*}: H_1(C_t,\mathbb{Z})\rightarrow H_1(N',\mathbb{Z}))$. Furthermore the corresponding   local system of vanishing degree $1$ cohomology  defined  on a Zariski open set of $\mathbb{P}^{n'-1}$ has no nonzero global section by Deligne's invariant cycle  theorem.

Recall that $P_{r}$  is a relative Picard variety for a fibration of $N'$ into curves, hence carries a universal relative divisor $D\in{\rm CH}^1(P_r\times_{\mathbb{P}^{n'-1}} N')$. Mapping the fibered product in the product $P_r\times N'$, this provides a correspondence $D\in {\rm CH}^{n'}(P_r\times N')$. Restricting $D$  to $N''\times N'$ gives  a correspondence  $\gamma\in {\rm CH}^{n'}(N''\times N')$.  The isomorphism  between ${\rm Alb}(N'')$ and $A\subset {\rm Alb}(N')$ described above  is the morphism $\gamma_*$  induced by the  correspondence $\gamma$.
It now  remains to construct the desired rational map $\psi$  from $N$ to $N''$ for an adequate choice of $s$. We  observe  that for two general points $x_1,\,x_2$ in $N'$, the element  ${\rm alb}(x_1-x_2)$ of ${\rm Alb}(N')$ has the property that
the set of points $m{\rm alb}_{N'}(x_1-x_2)$ is Zariski dense in ${\rm Alb}(N')$. If  $L$ is sufficiently ample, we can thus in the above construction choose $s=m  {\rm alb}_{N'}(x_1-x_2)$ for some $m$, where  $x_1,\,x_2$ are  two base-points of our linear system, since the right hand side is a general point of $A$.
Let now  $D$ be the degree of $L_t:=L_{\mid C_t}$. For an integer  $k$ large enough, we set $r=kD+1$ and we  get a rational
embedding
$$f: N'\dashrightarrow P_{r},
$$

$$x\mapsto x+k L_t +m(x_{1,t}-x_{2,t})_t\in {\rm Pic}^r(C_t) ,\,\,t:=f(x),$$
where $m,\,x_1,\,x_2$ are as above and  $x_{i,t}$ for $i=1,\,2$ is the point  $x_i$ seen as a point of the curve $C_t$.

The composite map
$$\psi=f\circ \phi$$
is well-defined  for a general choice of linear system, and we observe that
$\psi(N)\subset  N''=a_r^{-1}(A_s)$ since $A=\phi({\rm Alb}(M))$.
\end{proof}
  We conclude this section with the following  questions.

 \begin{question}  \label{questioninteressante} Does any rationally connected threefold admit a  universally generating family?
 \end{question}

   \begin{question} \label{questioninteressante2}  Does there exist a  rationally connected threefold  $X$ which admits    a universally generating family of   $1$-cycles, but does not admit a  universal codimension 2 cycle?
   \end{question}
   By \cite{voisinuniv}, an affirmative answer to Question \ref{questioninteressante}   implies an affirmative answer to Question \ref{questioninteressante2}.
\subsection{The case of cubic threefolds and quartic double solids}
Let $X$ be a smooth cubic threefold. One has ${\rm dim}\,J^3(X)=5$  and it is proved in \cite{voisinJAG} that if the minimal class $\frac{\theta^4}{4!}\in H^8(J^3(X),\mathbb{Z})$ is algebraic, then $X$ admits a  universal codimension $2$ cycle. This  implication is very specific to cubic threefolds but it only  uses  the fact that, by \cite{marku}, for any smooth  cubic threefold $X$, there exist  a smooth projective variety $M$, and a codimension $2$ cycle $Z\in{\rm CH}^2(M\times X)$ such that
$\Phi_Z:M\rightarrow J^3(X)$ is surjective  with rationally connected fibers.

Next it is proved in \cite{voisinJAG} that  the two conditions together (the algebraicity of the minimal class and the existence of a universal codimension $2$ cycle) imply the existence of a cohomological decomposition of the diagonal for a rationally connected threefold with no torsion in $H^3(X,\mathbb{Z})$.    Hence, in the case of the cubic threefold,  the algebraicity of the minimal class implies the existence of a cohomological decomposition of the diagonal. We used in turn in \cite{voisinuniv} these arguments to show that  a general desingularized quartic double solid with $7 $ nodes does not admit a parameterization
$$\Phi_Z: M\rightarrow J^3(X)$$
of its intermediate Jacobian with rationally connected fibers. Indeed, its minimal class is algebraic since it has dimension $3$, it has no torsion in $H^3$,  and it does not admit a cohomological decomposition of the diagonal by the main theorem of \cite{voisinuniv}.

 Coming back to the case of the cubic threefold $X$, it is even proved in  \cite{voisinjems} that if the minimal class $\frac{\theta^4}{4!}\in H^8(J^3(X),\mathbb{Z})$ is algebraic, then $X$ admits a Chow decomposition of the diagonal.
Conversely, the algebraicity of  the minimal class  of $J^3(X)$  is  implied by the existence of a cohomological decomposition of the diagonal of a rationally connected threefold $X$ (see \cite{voisinjems}), hence we conclude that for a cubic threefold, the algebraicity of the minimal class is equivalent to the existence of a  cohomological  (and even Chow-theoretic) decomposition of the diagonal.  Despite these results, the following question  remains open:

\begin{question} Let $X$ be a general cubic threefold. Does $X$ admit a universal codimension $2$ cycle?
\end{question}

Turning to the  property of the existence of a  universally generating family, Mingmin Shen proves the following.
\begin{theo}\label{theounigen}  (M. Shen, \cite{shen}.) Let $X$ be a smooth cubic threefold and $\Sigma_X$ the surface of lines in $X$. Then the universal family of lines in $X$ parameterized by $\Sigma_X$ is universally generating.
\end{theo}

The argument is quite interesting by its nonlinear character. Mingmin Shen associates to a $1$-cycle $Z\in {\rm CH}^2(X)$, the $0$-cycle of bisecant lines  to $Z$ and computes the class of the $1$-cycle of these lines in $X$.

In the case of a quartic double solid $X$, although we know that $X$ does not admit in general a universal codimension $2$ cycle, we do not know if $X$ admits a universally generating family (see Question \ref{questioninteressante}). In particular,  the following question seems to be  open.
\begin{question} Let $X$ be a smooth quartic double solid. Is the universal family of lines in $X$ universally generating?
\end{question}

    \end{document}